\documentclass[10pt,a4paper, reqno]{amsart}
\usepackage{mathrsfs}

\usepackage{amsmath,amsfonts,verbatim}
\usepackage{latexsym}
\usepackage{amssymb,leftidx}
\usepackage{extarrows}
\usepackage{overpic}
\usepackage{color}
\usepackage{epsfig}
\usepackage{subfigure}
\usepackage{tikz}
\usepackage{bbm}
\usepackage{tikz, caption}
\usepackage[font=small,labelfont=bf]{caption}

\usepackage[colorlinks=true,backref=page]{hyperref}
\hypersetup{urlcolor=blue, citecolor=blue}

\usepackage{amssymb}
\usepackage{mathrsfs}
\usepackage{amscd}
\usepackage{cite}

\newcommand\R{\mathbb{R}}

\newcommand\N{\mathbb{N}}

\newcommand\pa{\partial}
\newcommand\A{{\bf A}}
\newcommand{\ma}{\beta_2}

\newcommand\LL{\mathcal{L}}

\usepackage{graphicx}
\usepackage{float}

\numberwithin{equation}{section}
\newtheorem{proposition}{Proposition}[section]

\newtheorem{lemma}{Lemma}[section]
\newtheorem{theorem}{Theorem}[section]

\newtheorem{remark}{Remark}[section]

\begin{document}
\title[$L^p$-estimates for magnetic wave]{$L^p$-estimates for the wave equation with critical magnetic field in higher dimensions}

\author{Jialu Wang}
\address{Department of Mathematics, Beijing Institute of Technology, Beijing, China, 100081;} \email{jialu\_wang@bit.edu.cn}

\author{Chengbin Xu}
\address{School of Mathematics and Statistics, Qinghai Normal University, Xining, Qinghai, China, 810008; }
\email{xcbsph@163.com}

\author{Fang Zhang}
\address{Department of Mathematics, Beijing Institute of Technology, Beijing, China, 100081;} \email{zhangfang@bit.edu.cn}

\begin{abstract}
In this paper, we study the $L^{p}$-estimates for the solution to the wave equation with a scaling-critical magnetic potential in Euclidean $R^N$ with $N\geq3$. Inspired by the work of \cite{L}, we show that the operators $(I+\mathcal{L}_{\mathbf{A}})^{-\gamma}e^{it\sqrt{\mathcal{L}_{\mathbf{A}}}}$ is bounded in $L^{p}(\mathbb{R}^{N})$ for $1<p<+\infty$ when $\gamma>|1/p-1/2|$ and $t>0$, where $\mathcal{L}_{\mathbf{A}}$ is a magnetic Schr\"odinger operator. In particular, we derive the $L^{p}$-bounds for the sine wave propagator $\sin(t\sqrt{\mathcal{L}_{\mathbf{A}}})\mathcal{L}^{-\frac12}_{\mathbf{A}}$.
The key ingredient is the $L^p\rightarrow L^p$ boundedness of the analytic operator family $f_{w,t}(\mathcal{L}_{\mathbf{A}})$.
\end{abstract}
 \maketitle

\begin{center}
\begin{minipage}{120mm}
   { \small {{\bf Key Words:} $L^{p}$-estimates; Scaling-critical magnetic field; Wave equation.}
   }\\
\end{minipage}
\end{center}
\section{Introduction}
\subsection{The setting and motivation}
A fundamental question in harmonic analysis is to understand the $L^p$ boundedness of Fourier multipliers $M$ on $\mathbb{R}^N$, where these multipliers are defined by
$$
M(f)(x)=\frac{1}{(2 \pi)^n} \int_{\mathbb{R}^n} e^{i x \cdot \xi} m(\xi) \hat{f}(\xi) \,d \xi
$$
with $m$ as a measurable function. For radial multipliers where $m(\xi)=F(|\xi|)$, this analysis reduces to studying the $L^p$ boundedness of $F(\sqrt{\Delta})$, where $\Delta$ denotes the non-negative Laplacian in $\R^N$.
In particular, consider the metric measure space $X$ and a non-negative, self-adjoint operator $\boldsymbol{L}$ acting on the space $L^2(X)$. Such an operator allows for a spectral resolution $E_{\boldsymbol{L}}(\lambda)$, where $dE_{\boldsymbol{L}}(\lambda)$ represents the spectral measure connected with the operator $\boldsymbol{L}$. Therefore, given a real-valued Borel function $F$ on the interval $[0,\infty)$, we can define the operator $F(\boldsymbol{L})$ using the integral
\begin{equation}
F(\boldsymbol{L})=\int_{0}^{\infty}F(\lambda)\,dE_{\boldsymbol{L}}(\lambda).
\end{equation}
According to the spectral theory, the norm $\|F(\boldsymbol{L})\|_{L^2\rightarrow L^2}$ is bounded by the $L^{\infty}$ norm of $F$ over the spectrum of $\boldsymbol{L}$. However, the scenario for $p \neq 2$ presents significantly greater challenges.\vspace{0.2cm}

Over the last three decades, considerable significant attention has been devoted to identifying some necessary conditions on function $F$ that ensure $F(\boldsymbol{L})$ can extend to a bounded operator on a range of $L^p$ spaces for $p\neq2$.
Since the pioneering works of Mikhlin and H\"ormander \cite{H,Mi} on Fourier multipliers, we usually focus on finding conditions related to differentiability of the function $F$. These contributions have spurred significant progress in the study of spectral multipliers and broadened perspectives in harmonic analysis. The H\"ormander-Mikhlin theorem has been extended by numerous authors to a variety of operators beyond the Laplacian, as well as to frameworks outside the Euclidean setting. Probably the most natural and concrete examples are Schr\"odinger operators with real potential $\Delta+V$, or the Laplace-Beltrima operator on complete Riemannian manifolds. However these problems are also studied for abstract self-adjoint operators. In the theory of spectral multipliers, certain specific families of functions $F$ are also investigated including oscillatory integrals $e^{i(t \boldsymbol{L})^\alpha}\left(\mathrm{Id}+(t \boldsymbol{L})^\alpha\right)^{-\beta}$ and Bochner-Riesz means as the prominent examples. Given the extensive body of research on this topic, a comprehensive bibliography is impractical to provide here. For further reading, we refer the reader to \cite{COSY,Christ,Christ1,Christ2,Christ3,Co,DaSi,DSY,Fe,Fe1,Fe2,GHS,HaSi,Hebisch,DOS,SYY} and the references therein.\vspace{0.2cm}

 The objective of this article is to investigate certain $L^p$ multiplier properties associated with Schr\"odinger operator $\LL_{\A}$ with the scaling-critical magnetic potential on $\R^N$ for $N\geq3$, where the operator $\LL_{\A}$ given by
\begin{align}\label{LA}
\mathcal{L}_{\mathbf{A}}=\Big(i\nabla+\frac{\mathbf{A}(\hat{x})}{|x|}\Big)^{2},\quad x\in \mathbb{R}^{N}\backslash\{0\}.
\end{align}
Here $\hat{x}=\frac{x}{|x|}\in \mathbb{S}^{N-1}$ and $\mathbf{A}\in W^{1,\infty}(\mathbb{S}^{N-1};\R^N)$ satisfies the transversality condition
\begin{equation}\label{transcondition}
\mathbf{A}(\hat{x})\cdot \hat{x}=0,\quad \text{for all} \quad \hat{x}\in\mathbb{S}^{N-1}.
\end{equation}
Notably, the magnetic potential studied here is doubly critical, which are the scaling invariance and singularity of the potential. These features introduce additional  difficulties in various research fields, see \cite{FFFP,FFFP1,FZZ,GYZZ} and reference therein for dispersive and Strichatz estimates of Schr\"odinger, wave and Kelin-Gordorn equations; see \cite{BPST,FZZ1} for uniform resolvent estimates. Concerning the Schr\"odinger operator with magnetic potential, several results have been established, including Strichartz estimates, decay estimates, and local smoothing estimates. In this paper, we aim to establish the $L^p$ estimates for the wave propagator perturbed by the scaling-critical magnetic singular potentials.
\vspace{0.2cm}

Over the past few decades, extensive literature has developed on $L^p$-estimates for solution to the wave equation involving the Schr\"odinger operator (both in settings on manifolds and the operator with potentials), making it impractical to cite all relevant works here.
The study of $L^p$-estimate for the wave equation traces back to the space $\R^n$, classically going back to the work of \cite{M}\cite{P}. Notably, Peral \cite{P} and Miyachi \cite{M} established the sharp range of $p$, specifically $|\frac{1}{p}-\frac{1}{2}|\leq \frac{1}{n-1}$, for which $\frac{\sin(t\sqrt{-\Delta})}{\sqrt{-\Delta}}$ is $L^p(\R^n)$-bounded on $\R^n$. In addition, other $L^p\rightarrow L^q$ estimates were further explored, for example, by Strichartz \cite{S}. For the Hermite operator $-\Delta+|x|^2$, related results were presented in \cite{NT}, while broader operators of the form $-\Delta+V$ were investigated by Zhong \cite{Z}. Later, K. Jotsaroop and S. Thangavelu \cite{KT} established the operator $\frac{\sin\sqrt{G}}{G}$, in which $G$ denotes the Grushin operator $G=-\Delta-|x|^2\partial_{t}^2$, is $L^p(\R^{n+1})$-bounded for every $p$ satisfying $|\frac{1}{p}-\frac{1}{2}|<\frac{1}{n+2}$, thereby achieving the desired $L^p$ estimate for the corresponding solution. They subsequently derived the $L^p\rightarrow L^2$ estimate for wave equation with the Grushin operator via employing the $L^p$ boundedness of Riesz transform in \cite{TN}. Additionally, the $L^p$-boundedness of oscillating multipliers, on various classes of rank one locally symmetric spaces, was shown in \cite{E}.\vspace{0.2cm}

Specially, for Laplace-Beltrami operator on metric cone manifolds, Li and Lohoue in \cite{LN} established $L^p$-estimates for the wave equation on cone $C(Y)$, assuming that the injective radius $\rho$ of the closed section $Y$ exceeds $\pi$. Later, Li \cite{L} improved the result of \cite{LN} by removing the injective radius assumption for cone manifolds $C(Y)$ with the dimension $n\geq3$. Moreover, $L^p$-estimates for the wave equation have been extensively investigated in various other settings, including Heisenberg groups as explored in \cite{MS2}, symmetric spaces of non-compact type with real rank 1 as presented in \cite{GM, Ionescu}, and Damek-Ricci spaces as discussed in \cite{MV}.
For extra results on the wave equation in manifolds with conical singularities, refer to \cite{Mewun,MS1,SS1,SS2}. Inspired by the work of \cite{LN}, we continue to consider the the $L^p$-estimates for the wave equation involving critical magnetic Schr\"odinger operator $\mathcal{L}_{\bf A}$ in the metric cone manifold setting.

\subsection{The main result} In this paper, we investigate the $L^{p}$-estimates of the solution for the wave equations with a scaling-critical magnetic potential.
More precisely, we consider the following wave equation
\begin{align}\label{equation}
\begin{cases}
\partial_{tt}u+\mathcal{L}_{\mathbf{A}}u=0,& (t, x)\in \mathbb{R}\times \mathbb{R}^{N},\\
u(0, x)=f(x),& \partial_{t}u(0, x)=g(x),
\end{cases}
\end{align}
where the magnetic singular Schr\"odinger operator is given by \eqref{LA}.
It has been known that the solution to \eqref{equation} can be represented by
\begin{equation}
u(t,\cdot)=\cos(t\sqrt{\mathcal{L}_{\mathbf{A}}})f(\cdot)+\frac{\sin(t\sqrt{\mathcal{L}_{\mathbf{A}}})}{\sqrt{\mathcal{L}_{\mathbf{A}}}}g(\cdot).
\end{equation}
Our result related to the wave equation \eqref{equation} is concerned with the following.
\begin{theorem}\label{sinLA}
Let $\LL_{\A}$ be given magnetic Schr\"odinger operator in \eqref{LA}. There exists a constant $C>0$ such that, for any $1\leq p\leq+\infty,$
\begin{equation}\label{est:sinLA}
 \left\|\frac{\sin(t\sqrt{\mathcal{L}_{\mathbf{A}}})}{\sqrt{\mathcal{L}_{\mathbf{A}}}}g\right\|_{L^{p}(\mathbb{R}^{N})}\leq C |t|\cdot\Vert g\Vert_{L^{p}(\mathbb{R}^{N})},\quad \forall\; t\neq 0.
\end{equation}
\end{theorem}
In order to show Theorem \ref{sinLA}, we consider a family of analytic operators:
\begin{equation}
f_{w,t}(\mathcal{L}_{\mathbf{A}})=(\frac{\pi}{2})^{\frac{1}{2}}(t\sqrt{\LL_{\A}})^{w-\frac{N}{2}}J_{\frac{N}{2}-w}(t\sqrt{\LL_{\A}}),
\end{equation}
where $w=\varepsilon+(\frac{N+1}{2}-\varepsilon)z$ with $0\leq \mathrm{Re} z\leq1$ and $\frac{1}{2}<\varepsilon<1$.

Then, we have
\begin{theorem}\label{operator-family}
Let $\frac{1}{2}<\varepsilon<1$ and $ 1<p\leq+\infty$. Then, there exists a constant $C(\varepsilon,p)>0$ such that
\begin{equation}\label{est:fw}
\left\|f_{w,t}(\mathcal{L}_{\mathbf{A}})g\right\|_{L^p(\R^N)}\leq C(\varepsilon,p)\frac{e^{-2N\pi|y|}}{|\Gamma(\frac{N}{2}-w)|}\left\| g\right\|_{L^p(\R^N)}, \quad\forall\; g\in L^{p}(\mathbb{R}^{N}),
\end{equation}
where $w=\varepsilon+iy(\frac{N+1}{2}-\varepsilon)$ with $y\in \mathbb{R}$.
\end{theorem}
As a consequence of Theorem \ref{operator-family} and spectral multiplier theorem, we can acquire
\begin{theorem}\label{Lpestimate}
Let $\LL_{\A}$ be given magnetic Schr\"odinger operator in \eqref{LA} and let $\gamma>0$ and $1< p<+\infty$ satisfy $|\frac{1}{p}-\frac{1}{2}|<\frac{\gamma}{N-1}$. Then, there exists a constant $C(p, \gamma)>0$ such that for $t>0$ and for all $f\in L^{p}(\mathbb{R}^{N})$
\begin{equation}
\left\|(1+\mathcal{L}_{\mathbf{A}})^{-\frac{\gamma}{2}}e^{it\sqrt{\mathcal{L}_{\mathbf{A}}}}f\right\|_{L^{p}(\R^N)}\leq C(p, \gamma)(1+t)^{\gamma}\left\|f\right\|_{L^{p}(\R^N)}.
\end{equation}
\end{theorem}

\section{Preliminaries}\label{Preliminaries}
\subsection{Notations} We denote the Gamma function as $\Gamma(z)$. Also, we use $P_{\lambda}^{\mu}(x)$ and $Q_{\lambda}^{\mu}(z)$ to represent the Legendre functions when $|x|<1$ and $|z|>1$ respectively, for more details and properties about $P_{\lambda}^{\mu}(x)$ and $Q_{\lambda}^{\mu}(z)$, please refer to \cite{L,LN,W}.
\subsection{The kernel of analytic operator $f_{w,t}(\LL_{\A})$}\label{functional}
The purpose of this subsection is to demonstrate the kernel of analytic operator $f_{w,t}(\LL_{\A})$, to this end, we first recall the spectral property for the magnetic operator $\LL_{\A}$ with $N\geq3$ and classical functional calculus.

Recall the magnetic operator $\mathcal{L}_{\mathbf{A}}$ in \eqref{LA} given by
\begin{align}\label{def:LA}
\mathcal{L}_{\mathbf{A}}=-\Delta+|x|^{-2}(|\mathbf{A}(\hat{x})|^2+i \text{div}_{\mathbb{S}^{N-1}}\mathbf{A}(\hat{x}))+2i|x|^{-1}\mathbf{A}(\hat{x})\cdot \nabla.
\end{align}
Using \eqref{transcondition} and polar coordinates, we can write $\mathcal{L}_{\mathbf{A}}$ as follows
\begin{equation}\label{def:LbfA} \mathcal{L}_{\mathbf{A}}=-\partial_{r}^{2}-\frac{1}{r}\partial_{r}+\frac{L_{\mathbf{A}}}{r^{2}},  \end{equation}
in which the operator $L_{\mathbf{A}}$ is denoted by
\begin{align}
L_{\mathbf{A}}&=(i\nabla_{\mathbb{S}^{N-1}}+\mathbf{A}(\hat{x}))^{2},\quad \hat{x}\in \mathbb{S}^{N-1}\nonumber\\
&=-\Delta_{\mathbb{S}^{N-1}}+(|\mathbf{A}(\hat{x})|^{2}+i\mathrm{d}\mathrm{i}\mathrm{v}_{\mathbb{S}^{N-1}}\mathbf{A}(\hat{x}))+2i\mathbf{A}(\hat{x})\cdot\nabla_{\mathbb{S}^{N-1}}.
\end{align}
According to the classical spectral theory, $\LL_{\mathbf{A}}$ possesses a diverging sequence of real eigenvalues with finite multiplicity satisfying $\mu_1(\mathbf{A}) \leq \mu_2(\mathbf{A}) \leq \cdots \leq \mu_k(\mathbf{A}) \leq \cdots$, more details about this operator $\LL_{\A}$ can be found in \cite{FFFP} and \cite[Lemma 2.5]{FFT}.
For each $k \in \mathbb{N}, k \geqslant 1$, we write $\psi_k$ as a $L^2\left(\mathbb{S}^{N-1}, \mathbb{C}\right)$-normalized eigenfunction of the operator $L_{\mathbf{A}}$ on sphere $\mathbb{S}^{N-1}$, which correspond to the $k^{\text {th }}$ eigenvalue $\mu_k(\mathbf{A})$. Then, we have
\begin{align}\label{eigenfunction}
\begin{cases}
L_{\mathbf{A}} \psi_k=\mu_k(\mathbf{A}) \psi_k(\theta), \quad \text { in } \mathbb{S}^{N-1},\\
\int_{\mathbb{S}^{N-1}}\left|\psi_k(\theta)\right|^2\; d S(\theta)=1.
\end{cases}
\end{align}
Specially, in the sequence of eigenvalues $\mu_1(\mathbf{A}) \leqslant \mu_2(\mathbf{A}) \leq \cdots \leq \mu_k(\mathbf{A}) \leq \cdots$, each eigenvalue is repeated according to its multiplicity. Thus, there is exactly one eigenfunction $\psi_k$ corresponding to each index $k\in \mathbb{N}$. We can select the functions $\psi_k$ in such a way that they form an orthonormal basis for $L^2\left(\mathbb{S}^{N-1}, \mathbb{C}\right)$.
Under the condition $\mu_{1}>-(\frac{N-2}{2})^2$, the quadratic form connected with $\LL_{\A}$ is positive definite, then meaning that $\LL_{\A}$ is a symmetric and semi-bounded operator on $L^2\left(\mathbb{S}^{N-1}, \mathbb{C}\right)$. As a result, it admits a self-adjoint extension, specifically the Friedrichs extension, which has the natural form domain.

Now, in order to demonstrate Theorem \ref{operator-family}, we start by recalling the kernel of the analytic operator $f_{w,t}(\LL_{\A})$. For simplicity, we define the following notations:
\begin{align*}
\alpha=&\frac{N-2}{2},\quad L=\sqrt{L_{\mathbf{A}}+\alpha^2},\\
w=&\varepsilon+iy\big(\tfrac{N+1}2-\varepsilon\big)\;\quad\text{with}\;
\tfrac12<\varepsilon<1\;\text{and}\;y\in\R,
\end{align*}
and for $h>0$, we introduce the operator
$$J_{L}(h)=\int_0^\infty J_{\sqrt{\xi+\alpha^2}}(h)\;dE_{L}(\xi).$$
Next, using the polar coordinates, we can express $x,y\in\R^{N}$ as follows
$$x=(r_1\cos\theta_1,r_1\sin\theta_1)\triangleq(r_1,\theta_1),$$
and
$$y=(r_2\cos\theta_2,r_2\sin\theta_2)\triangleq(r_{2},\theta_{2}).$$

Let the kernel of operator $f_{w,t}(\LL_{\A})$ be denoted by $K_{w,t}((r_1,\theta_1),(r_2,\theta_2))$.
Then, from the standard functional calculus for $\LL_{\A}$ on $\R^N$ established in \cite{CT}, we can have
\begin{align}\label{kernel}
&K_{w,t}((r_1,\theta_1),(r_2,\theta_2))\nonumber\\
=&(r_1r_2)^{-\alpha}\int_0^\infty f_{w,t}(\lambda^2)J_L(\lambda r_1)J_{L}(\lambda r_2)\lambda\;d\lambda\nonumber\\
=&\frac{1}{2}t^{2(w-\frac{n}{2})}(r_1r_2)^{-w}\times \begin{cases}
0&\text{if}\quad t<|r_1-r_2|;\\
\sin^{\frac{N-1}{2}-w}\beta_1 P_{L-\frac{1}{2}}^{w-\frac{N-1}{2}} &\text{if}\quad |r_1-r_2|<t<r_1+r_2;\\
\frac{2}{\pi}\sinh^{\frac{N-1}{2}-w}\beta_2\cos(\pi L)Q_{L-\frac{1}{2}}^{w-\frac{N-1}{2}}(\cosh\beta_2) & \text{if} \quad t>r_1+r_2;\end{cases}\nonumber\\
=&\frac{1}{\sqrt{2\pi}\Gamma(\frac{N}{2}-w)}t^{2(w-\frac{N}{2})}(r_1r_2)^{-w}
\nonumber\\
&\times \begin{cases}
0 &\text{if}\quad t<|r_1-r_2|;\\
\int_0^{\beta_1}(\cos h-\cos\beta_1)^{\alpha-w}\cos(hL)(\theta_1,\theta_2)\;dh &\text{if}\quad |r_1-r_2|<t<r_1+r_2;\\
e^{i\pi(w-\frac{N-1}{2})}\int_{\beta_2}^\infty (\cos h-\cos \beta_2)^{\alpha-w}e^{-hL}\cos(\pi L)(\theta_1,\theta_2)\;dh & \text{if} \quad t>r_1+r_2;
\end{cases}
\end{align}
where
$$\beta_1=\arccos\frac{r_1^2+r_2^2-t^2}{2r_1r_2}\quad \text{and}\quad
\beta_2={\rm arc}\cosh\frac{t^2-r_1^2-r_2^2}{2r_1r_2}.$$
\subsection{Hadamard parametrix}It is essential to use the fundamental solution when calculating the $p\to p$ boundedness of the analytic operator family in Proposition \ref{prop:I2}, so we require the results concerning the fundamental solution to the wave equation on compact manifold, which could be found in \cite{Ber,Ber90,Hor}. In the context of this article, our compact manifold is $\mathbb{S}^{N-1}$. Consequently, in the following lemma statement, we will record exclusively the fundamental solution to wave equation on the compact manifold $\mathbb{S}^{N-1}$.
\begin{lemma}\label{parametrix}
For any $K\in\N^\ast,\;0<h<A_0(<\rho_0)$ and $s>0$, where $\rho_0$ denotes the injective radial of $\mathbb{S}^{N-1}$, we have (formally):
\begin{equation}\label{equ:cosslmm}
  \cos(sL)(\theta_1,\theta_2)=\sum_{k=0}^{K}U_k(\theta_1,\theta_2)s\frac{(s^2-d_{\mathbb{S}^{N-1}}^2(\theta_1,\theta_2))_+^{k-\frac{N}{2}}}
  {\Gamma(k-\frac{N}{2}+1)}+sC_k(s,\theta_1,\theta_2),
\end{equation}
where

$\bullet$ $U_k(\theta_1,\theta_2)\in C^\infty(\mathbb{S}^{N-1}\times\mathbb{S}^{N-1}).$

$\bullet$ $C_k(s,\theta_1,\theta_2)\in C([0,A_0]\times \mathbb{S}^{N-1}\times \mathbb{S}^{N-1})$, for $K$ sufficient large, for example $K=N+1$, and $C_K(s,\theta_1,\theta_2)=0$ when $d_{\mathbb{S}^{N-1}}(\theta_1,\theta_2)>s$.

$\bullet$ For $g\in C^\infty(\mathbb{S}^{N-1})$ and $\theta_1\in \mathbb{S}^{N-1}$, we note
$$\overline{g_{\theta_1}}(h)=\int_{S^h(\theta_1)}g(\theta_1^{'})\;d\sigma(\theta_1^{'}),\quad 0<h<\rho_0,$$
where $S^h(\theta_1)=\{\theta_1{'}\in \mathbb{S}^{N-1};\;d_{\mathbb{S}^{N-1}}(\theta_1,\theta_1^{'})=h\}$ and $d\sigma(\theta_1^{'})$ denotes the induced surface measure. Then, for $\tfrac{N}2-k\geq1$, in the distribution sense $\frac{(s^2-d_{\mathbb{S}^{N-1}}^2(\theta_1,\theta_2))_+^{k-\frac{N}{2}}}
{\Gamma(k-\frac{N}{2}+1)}$ (for $0<s<\rho_0$) is as follows:
\begin{align*}
&\Big\langle \frac{(s^2-d_{\mathbb{S}^{N-1}}^2(\theta_1,\theta_2))_+^{k-\frac{N}{2}}}
{\Gamma(k-\frac{N}{2}+1)},g\Big\rangle(\theta_1)=\Big\langle
\frac{(s^2-h^2)_+^{k-\frac{N}{2}}}
{\Gamma(k-\frac{N}{2}+1)},\overline{g_{\theta_1}}(h)\Big\rangle\\
=& \begin{cases}
(\frac{1}{2s}\frac{\pa}{\pa s})^{\frac{N}{2}-k-1}[\frac{1}{2s}\overline{g_{\theta_1}}(s)],\quad\text{if}\quad
\frac{N}{2}\in\N\\
(\frac{1}{2s}\frac{\pa}{\pa s})^{\frac{N}{2}-k-\frac12}\int_0^s\frac{(s^2-h^2)^{-\frac12}}{\Gamma(1/2)}
\overline{g_{\theta_1}}(h)\;dh\quad\text{if}\quad
\frac{N}{2}\not\in\N
\end{cases}
\end{align*}
for all $\theta_1\in \mathbb{S}^{N-1}$ and $g\in C^\infty(\mathbb{S}^{N-1})$.
\end{lemma}
\section{The proof of Theorem \ref{operator-family}}
In this section, we mainly focus on the proof of Theorem \ref{operator-family} inspired by the method in \cite{L}. We divide the kernel of the analytic operator $f_{w,t}(\LL_{A})$ into three parts: $I_1(w,t)$, $I_2(w,t)$ and $I_3(w,t)$, thus the proof of Theorem \ref{operator-family} boils down to verifying the following three lemmas.

In the following, we first fix $A_0=10^{-3}\min\{1,\rho_0\}$, then we have
\begin{lemma}\label{prop:I1}
There exists a constant $C(\varepsilon)>0$ such that for all $1\leq p\leq+\infty$
\begin{equation}\label{equ:t1est}
\big\|I_1(w,t)f\big\|_p\leq C(\varepsilon)\|f\|_p,\quad \forall\;f\in L^p(\R^N),
\end{equation}
where the integral kernel $I_1(w,t)$ is defined by
\begin{equation}\label{equ:t1wt}
I_1=t^{2(w-\frac{N}{2})}(r_1r_2)^{-w}\chi_{\{t>r_1+r_2\}}e^{i\pi(w-\frac{N-1}{2})}
\int_{\beta_2}^\infty (\cosh h-\cosh\beta_2)^{\alpha-w}
e^{-hL}\cos(\pi L)\;dh.
\end{equation}
\end{lemma}

\begin{lemma}\label{prop:I2}
There exists a constant $C(\varepsilon)>0$ such that for all $1\leq p\leq+\infty$
\begin{equation}\label{equ:t21est}
\big\|I_{2}(w,t)f\big\|_p\leq C(\varepsilon)\|f\|_p,\quad \forall\;f\in L^p(\R^N),
\end{equation}
where the integral kernel $I_2(w,t)$ is denoted by
\begin{align}\label{proof:I2}
I_2=t^{2(w-\frac{N}{2})}(r_1r_2)^{-w}\chi\Big\{\cos
\beta_1=\tfrac{r_1^2+r_2^2-t^2}{2r_1r_2}>\cos A_0\Big\}\nonumber\\
\quad\quad\times\int_0^{\beta_1}(\cos h-\cos \beta_1)^{\alpha-w}\cos(hL)(\theta_1,\theta_2)\;dh.
\end{align}
\end{lemma}

\begin{lemma}\label{prop:I3}
For all $1<p<+\infty,$ there exists a constant $C(p,\varepsilon)>0$ such that
\begin{equation}\label{equ:t22est}
\big\|I_{3}(w,t)f\big\|_p\leq C(p,\varepsilon)\frac{e^{2n\pi|y|}}{|\Gamma(\frac{n}{2}-w)|}\|f\|_p,\quad
\forall\;f\in L^p(\R^N),
\end{equation}
in which the integral kernel $I_3(w,t)$ is written by
\begin{equation}
I_3=t^{2(w-\frac{N}{2})}(r_1r_2)^{-w}\chi\Big\{
-1<\cos \beta_1=\tfrac{r_1^2+r_2^2-t^2}{2r_1r_2}\leq\cos A_0\Big\}
\sin^{\frac{N-1}{2}-w}\beta_1P_{L-\frac12}^{w-\frac{N-1}{2}}(\cos \beta_1).
\end{equation}
\end{lemma}
Indeed, from \eqref{kernel}, we observe that
\begin{equation}\label{equ:fwtt1t21}
f_{w,t}(\LL_{\A})=\frac{1}{\sqrt{2\pi}\Gamma(\frac{N}{2}-w)}\big[I_1(w,t)+I_{2}(w,t)\big]
+\frac12I_{3}(w,t).
\end{equation}
Hence, it follows that the validity of three lemmas above is essential. Once these lemmas are proven, Theorem \ref{operator-family} can be derived immediately. Now, we turn our attentions to showing Lemma \ref{prop:I1}, Lemma \ref{prop:I2} and Lemma \ref{prop:I3} respectively.
\subsection{Proof of Lemma \ref{prop:I1}}
Before we begin to prove the Lemma \ref{prop:I1}, it is necessary to estimate the kernel of the operator $e^{-hL}\cos(\pi L)(h>0)$, so we first establish the following proposition which is crucial to demonstrate Lemma \ref{prop:I1}.
\begin{proposition}\label{proof:I1}
Let $H_h(\theta_1,\theta_2)$ be the kernel of the operator $e^{-hL}\cos(\pi L)(h>0)$. Set
\begin{equation}\label{equ:ehlsubast}
\big\|e^{-hL}\cos(\pi L)\big\|_\ast:=\sup_{\theta_1\in\mathbb{S}^{N-1}}\int_{\mathbb{S}^{N-1}}\big|H_h(\theta_1,\theta_2)\big|\;d\theta_2,
\end{equation}
and
\begin{equation}\label{equ:ehlsupast}
\big\|e^{-hL}\cos(\pi L)\big\|^\ast:=\sup_{\theta_2\in\mathbb{S}^{N-1}}\int_{\mathbb{S}^{N-1}}\big|H_h(\theta_1,\theta_2)\big|\;d\theta_1.
\end{equation}
Then, there exists a constant $C>0$ such that for all $1\leq p\leq+\infty$ and $h>0$, there holds
\begin{align}\nonumber
 \big\|e^{-hL}\cos(\pi L)\big\|_{p\to p} \leq&\max\big\{  \big\|e^{-hL}\cos(\pi L)\big\|_\ast, \big\|e^{-hL}\cos(\pi L)\big\|^\ast\big\}\\\label{equ:eppcosest}
 \leq& Ce^{-h\alpha}\big(1+h^{\frac{1-n}{2}}\big).
\end{align}
Therefore, for all $\beta_2>0$, we also can obtain
\begin{align}\nonumber
& \int_{\beta_2}^\infty\big[\big\|e^{-hL}\cos(\pi L)\big\|_\ast+\big\|e^{-hL}\cos(\pi L)\big\|^\ast\big]\cdot
\big|(\cosh h-\cosh \beta_2)^{\alpha-w}\big|\;dh\\\label{equ:kerahaig}
\leq&C\cdot (\cosh \beta_2-1)^{-{\rm Re}w}.
\end{align}
\end{proposition}
\begin{proof}
We first show \eqref{equ:eppcosest}. By spectral theorem, we can write
$$H_h(\theta_1,\theta_2)=\sum_{j=0}^{+\infty}\cos\big(\pi\sqrt{\lambda_j+\alpha^2}
\big)e^{-h\sqrt{\lambda_j+\alpha^2}}\psi_j(\theta_1)\overline{\psi_j(\theta_2)},$$
where $0=\lambda_0<\lambda_1\leq \lambda_2\leq\cdots$ is the sequence of the eigenvalues of $L_{\A}$ and $\psi_j$ is a normalized eigenfunction corresponding to $\lambda_j$, see \eqref{eigenfunction} for detailed discussion.

We remark that $H_h(\theta_1,\theta_2)=\overline{H_h(\theta_2,\theta_1)}$ and $\mathbb{S}^{N-1}$ is compact manifold, from Shur'test Lemma, it suffices to certify that
$$\sup_{\theta_1\in \mathbb{S}^{N-1}}\int_{\mathbb{S}^{N-1}}\big|H_h(\theta_1,\theta_2)\big|^2\;d\theta_2<\big\{C\cdot e^{-h\alpha}\big[1+h^{\frac{1-n}{2}}\big]\big\}^2.$$
Note that
\begin{equation}
e^{-yA}=\frac{1}{2}y\pi^{-\frac{1}{2}}\int_{0}^{\infty}e^{-\frac{y^2}{4t}}e^{-tA^2}t^{-\frac{3}{2}}\;dt,
\end{equation}
and
\begin{equation}
e^{tL_{A}}(\theta_1,\theta_1)\leq t^{-\frac{N-1}{2}},
\end{equation}
where $e^{uL_{A}}(\theta_1,\theta_1)$ satisfies the heat kernel estimate on sphere $\mathbb{S}^{N-1}$ from the Weyl's formula.
Then, we can acquire
\begin{align*}
\int_{\mathbb{S}^{N-1}}\big|H_h(\theta_1,\theta_2)\big|^2\;d\theta_2\leq&\sum_{j=0}^{+\infty}
\cos^2\big(\pi\sqrt{\lambda_j+\alpha^2}\big)e^{-2h\sqrt{\lambda_j+\alpha^2}}
|\psi_j(\theta_1)|^2\\
\leq&e^{-\alpha h}\sum_{j=0}^{+\infty}e^{-h\sqrt{\lambda_j+\alpha^2}}|\psi_j(\theta_1)|^2\\
=&e^{-\alpha h}e^{-h\sqrt{L_{A}+\alpha^2}}(\theta_1,\theta_1)\\
=&e^{-\alpha h}\frac{h}{2\sqrt{\pi}}\int_0^\infty u^{-\frac32}e^{-\frac{h^2}{4u}}e^{-\alpha^2u}e^{u L_{A}}(\theta_1,\theta_1)\;du\\
\leq&C\cdot e^{-2h\alpha}\big[1+h^{-(N-1)}\big],
\end{align*}
which implies \eqref{equ:eppcosest}.

Next, we focus on verifying \eqref{equ:kerahaig}. By ${\rm Re}w=\varepsilon,$ and \eqref{equ:eppcosest}, it suffices to show that there exists a constant $C>0$ such that for all $\beta_2>0$, the following inequality holds
\begin{equation}\label{equ:omegest}
F:=\int_{\beta_2}^\infty e^{-\alpha h}(1+h^\frac{1-n}{2})(\cosh h-\cosh\beta_2)^{\alpha-\varepsilon}\;dh\leq C\cdot(\cosh \beta_2-1)^{-\varepsilon}.
\end{equation}
We will now proceed to establish the above estimate in the two cases: $\ma\geq1$ and $\ma<1$.

{\bf Case 1: $\ma\geq1$.} First, we note that
\begin{equation}
\mathcal{D}_{\bar{w}}^\ell=e^{i\pi\ell}2^{-\bar{w}-1}
\sqrt{\pi}\frac{\Gamma(\bar{w}+\ell+1)}{\Gamma(\frac{3}{2}+\bar{w})}
z^{-\bar{w}-1}\big(1+o(1)\big),\quad\forall\;|z|\gg1,
\end{equation}
where $\mathcal{D}_{\bar{w}}^\ell$ denotes Legendre function of the second kind, see \cite[p.186]{Mag}. Therefore, along with the fact that $\tfrac12<\varepsilon<1$ and
$N\geq3$, we can compute
\begin{align*}
F\leq & 2\int_{\ma}^\infty e^{-h\alpha}(\cosh h-\cosh\ma)^{\alpha-\varepsilon}\;dh\\
=&2\big(\tfrac{\pi}2\big)^{-\frac12}e^{-i\pi(\varepsilon-\frac{N-1}{2})}
\Gamma\big(\tfrac{N}2-\varepsilon\big)(\sinh \ma)^{\frac{N-1}{2}-\varepsilon}\mathcal{D}_{\alpha-\frac12}^{\varepsilon-\frac{N-1}{2}}
(\cosh\ma)\\
\leq & C\cdot(\sinh \ma)^{\frac{N-1}{2}-\varepsilon}(\cosh \ma)^{-\frac{N-1}{2}}\\
\leq&C\cdot(\cosh \ma)^{\frac{N-1}{2}-\varepsilon}(\cosh \ma)^{-\frac{N-1}{2}}\\
\leq&C\cdot(\cosh\ma-1)^{-\varepsilon}.
\end{align*}
This concludes the proof of case 1 where $\beta_2\geq1$.

{\bf Case 2: $0<\ma<1.$} This case is more complicated than the previous one where $\beta_2\geq1$. According to the interval of integral, we divide $F$ into three terms: $F_1$, $F_2$ and $F_3$, which are respectively represented by
\begin{equation}
F_1\triangleq\int_{\ma}^{2\ma}e^{-\alpha h}(1+h^\frac{1-N}{2})(\cosh h-\cosh\ma)^{\alpha-\varepsilon}\;dh,
\end{equation}
\begin{equation}
F_2\triangleq\int_{2\ma}^{1+\ma}e^{-\alpha h}(1+h^\frac{1-N}{2})(\cosh h-\cosh\ma)^{\alpha-\varepsilon}\;dh,
\end{equation}
and
\begin{equation}
F_3\triangleq\int_{1+\ma}^{\infty} e^{-\alpha h}(1+h^\frac{1-N}{2})(\cosh h-\cosh\ma)^{\alpha-\varepsilon}\;dh.
\end{equation}
Therefore, the goal of subsequent analysis is to estimate the three terms: $F_1$, $F_2$ and $F_3$.

We first conosider the term $F_3$. Set $h=\beta_2+s$, then we can achieve
\begin{align*}
F_3 \lesssim& 2\int_{1+\ma}^\infty  e^{-\alpha h}(\cosh h-\cosh\ma)^{\frac{N}{2}-\varepsilon}\;dh\\
=&2e^{-\alpha\ma}\int_1^\infty e^{-\alpha s}\big(\cosh(\ma+s)-\cosh\ma\big)^{\alpha-\varepsilon}\;ds\\
=&2e^{-\alpha\ma}\int_1^\infty e^{-\alpha s}\Big[2\sinh\frac{2\ma+s}{2}\sinh\frac{s}{2}\Big]^{\alpha-\varepsilon}\;ds\\
\lesssim&C\int_1^\infty e^{-\alpha s}[e^s]^{\alpha-\varepsilon}\;ds\\
\lesssim&C\frac{1}{\varepsilon}\leq C\cdot\big(\cosh\ma-1\big)^{-\varepsilon}.
\end{align*}
Moreover, when $0<\ma\leq1$, we observe that
$$\ma^2\sim 4\sinh^2\frac{\ma}{2}=2\big(\cosh\ma-1\big).$$
Hence, by $\frac12<\varepsilon<1$, we obtain
\begin{align*}
F_2\lesssim & C\int_{2\ma}^{1+\ma}h^{1-N}\Big(2\sinh\frac{h+\ma}{2}\cdot\sinh\frac{h-\ma}{2}\Big)^{\alpha-\varepsilon}\;dh\\
\lesssim&C\int_{2\ma}^{1+\ma}h^{1-N}\cdot(h^2)^{\alpha-\varepsilon}\;dh\\
\lesssim&C\frac{1}{\varepsilon}\ma^{-2\varepsilon}
\lesssim C\cdot\left(\cosh\ma-1\right)^{-\varepsilon},
\end{align*}
provided that $\sinh\frac{h+\ma}{2}\sim\frac{h+\ma}{2}\sim h$ and $\sinh\frac{h-\ma}{2}\sim\frac{h-\ma}{2}\sim h$ in the case $2\ma\leq h\leq1+\ma$.

On the other hand, we can calculate
\begin{align*}
F_1\lesssim & C\int_{\ma}^{2\ma}h^{1-N}\Big(2\sinh\frac{h+\ma}{2}\cdot\sinh\frac{h-\ma}{2}\Big)^{\alpha-\varepsilon}\;dh\\
\lesssim&C\ma^{1-N}\ma^{\alpha-\varepsilon}\int_{\ma}^{2\ma}(h-\ma)^{\alpha-\varepsilon}\;dh\\
=&C\frac{1}{\frac{N}{2}-\varepsilon}\cdot\ma^{-2\varepsilon}\\
\lesssim&C\cdot\big(\cosh\ma-1\big)^{-\varepsilon}.
\end{align*}
in which we used the fact that $\sinh\frac{h+\ma}{2}\sim\frac{h+\ma}{2}\sim\ma$ and $\sinh\frac{h-\ma}{2}\sim\frac{h-\ma}{2}$ when $\ma\leq h\leq2\ma$.
Therefore, we complete the proof of \eqref{equ:kerahaig} and Proposition \ref{proof:I1} is valid.
\end{proof}
\textbf{Proof of Lemma \ref{prop:I1}:}
Now, we begin to prove Lemma \ref{prop:I1}. Let $I_1\big((r_1,\theta_1),(r_2,\theta_2)\big):=I_1(w,t)\big((r_1,\theta_1),(r_2,\theta_2)\big)$ be defined by \eqref{equ:t1wt}. Based on Riesz's convex theorem, it is sufficient to demonstrate that there exists a constant $C(\varepsilon)>0$ such that for all $t>0$ and $w=\varepsilon+iy\big(\tfrac{N+1}2-\varepsilon\big)$, the following calculations hold
\begin{align}\label{equ:aim}
\sup_{(r_1,\theta_1)\in \R^N}\int_{\R^N}\big|I_1\big((r_1,\theta_1),(r_2,\theta_2)\big)\big|\;dy<C(\varepsilon),
\end{align}
\begin{align}
\sup_{(r_2,\theta_2)\in \R^N}\int_{\R^N}\big|I_1\big((r_1,\theta_1),(r_2,\theta_2)\big)\big|\;dx<C(\varepsilon).
\end{align}

Observe that
$$\big|I_1\big((r_1,\theta_1),(r_2,\theta_2)\big)\big|
=\big|I_1\big((r_2,\theta_2),(r_1,\theta_1)\big)\big|,$$
thus, we only need to consider \eqref{equ:aim}. In fact, for $x,y\in \R^N$ and $t>r_1+r_2$, a straightforward calculation yields
\begin{align*}
&\int_{\R^N}\big|I_1\big((r_1,\theta_1),(r_2,\theta_2)\big)\big|\;dy\\
=&\int_{0}^{\infty}\Big\{\int_{\mathbb{S}^{N-1}}\Big|t^{2(w-\frac{N}{2})}(r_1r_2)^{-w}
\int_{\beta_2}^\infty (\cosh h-\cosh\beta_2)^{\alpha-w}
e^{-hL}\cos(\pi L)\;dh\Big|\times \;d\theta_2\Big\}r_2^{N-1}\;dr_2\\
\lesssim&\int_{0}^{\infty} t^{2({\rm Re}w-\frac{N}{2})}(r_1r_2)^{-{\rm Re}w}\Big[\int_{\beta_2}^\infty\big\|e^{-hL}\cos(\pi L)\big\|_\ast(\cosh h-\cosh \beta_2)^{\alpha-{\rm Re}w}\;dh\Big] r_2^{N-1}\;dr_2\\
\lesssim&C\int_{0}^{\infty} t^{2({\rm Re}w-\frac{N}{2})}(r_1r_2)^{-{\rm Re}w}\cdot(\cosh\beta_2-1)^{-{\rm Re}w}r_2^{N-1}\;dr_2\\
=&C(\varepsilon)\int_{0}^{\infty}t^{2(\varepsilon-\frac{N}{2})}
\big[t^2-(r_1+r_2)^2\big]^{-\varepsilon}r_2^{N-1}\;dr_2\\
\lesssim&t^{2(\varepsilon-\frac{N}{2})}t^{N-1}t^{-\varepsilon}\int_0^{t-r_1}
\big[t-(r_1+r_2)\big]^{-\varepsilon}\;dr_2
\lesssim\frac{1}{1-\varepsilon}.
\end{align*}
where we employed the fact that $\beta_2={\rm arc}\cosh\frac{t^2-r_1^2-r_2^2}{2r_1r_2}$, as well as the condition $\tfrac12<\varepsilon<1$ and Proposition \ref{proof:I1}.

So far, we have finished the proof of Lemma \ref{prop:I1}.

\subsection{Proof of Lemma \ref{prop:I2}}
The purpose of this section is to deduce Lemma \ref{prop:I2}. Unlike the previous proof, this requires the use of fundamental solutions, thus making the situation more intricate. The overall proof strategy is similar to that in \cite{L}. For the sake of completeness and fluency of the article, we will provide the detailed proof once more.

In order to show Lemma \ref{prop:I2}, we start by estimating its kernel.
\begin{proposition}\label{lem:I2}
Let $I_{2}\big((r_1,\theta_1),(r_2,\theta_2)\big)$ be given by \eqref{proof:I2} with ${\rm Re}w=\varepsilon\in\big(\tfrac12,1\big).$ Then, there exists a constant $C(N,\varepsilon)>0$ such that
\begin{align}\nonumber
\big| I_{2}\big((r_1,\theta_1),(r_2,\theta_2)\big) \big|\leq & C(N,\varepsilon
)t^{2(\varepsilon-\frac{N}{2})}(r_1r_2)^{-\varepsilon}\chi\left\{r_1,r_2>0;
\cos \beta_1=\tfrac{r_1^2+r_2^2-t^2}{2r_1r_2}>\cos
A_0\right\}\\\label{equ:t21kernesta}
&\times \left(\beta_1 ^2-d_{\mathbb{S}^{N-1}}^2(\theta_1,\theta_2)\right)^{-\varepsilon}\chi\left\{\theta_1,\theta_2\in \mathbb{S}^{N-1}; d_{\mathbb{S}^{N-1}}(\theta_1,\theta_2)<\beta_1\right\}.
\end{align}
\end{proposition}

\begin{proof}
By the definition of $I_{2}\big((r_1,\theta_1),(r_2,\theta_2)\big)$ given by \eqref{proof:I2}, we only need to consider
$$G=\int_0^{\beta_1}\big(\cos h-\cos \beta_1\big)^{\alpha-w}\cos(hL)(\theta_1,\theta_2)\;dh,\quad \beta_1<A_0.$$
According to the parametrix of the wave operator on $\mathbb{S}^{N-1}$ and \eqref{equ:cosslmm}, it is sufficient to estimate the following two terms
\begin{align*}
&G_{e}= \int_0^{\beta_1}\big(\cos h-\cos \beta_1\big)^{\alpha-w}hC_{N+1}(h,\theta_1,\theta_2)\;dh,\\
&G_k=\int_0^{\beta_1}\big(\cos h-\cos \beta_1\big)^{\alpha-w}h\frac{(h^2-d_{\mathbb{S}^{N-1}}(\theta_1,\theta_2)^2)_+^{k-\frac{N}{2}}}{\Gamma(k-\frac{N}2+1)}
\;dh,\quad0\leq k\leq N+1.
\end{align*}
We first focus on demonstrating the simpler term $G_e$, and the error term $G_e$ can be obtained directly
\begin{align*}
\big| G_{e}\big|\leq& \chi\big\{\theta_1,\theta_2\in \mathbb{S}^{N-1};\;d_{\mathbb{S}^{N-1}}(\theta_1,\theta_2)<\beta_1\big\}\int_{d_{\mathbb{S}^{N-1}}(\theta_1,\theta_2)}^{\beta_1}\big(\cos h-\cos \beta_1\big)^{\alpha-\varepsilon}h\;dh\\
\leq&C\chi\big\{\theta_1,\theta_2\in \mathbb{S}^{N-1};\;d_{\mathbb{S}^{N-1}}(\theta_1,\theta_2)<\beta_1\big\},
\end{align*}
in which we utilized the condition $N\geq3,\tfrac12<\varepsilon<1$ and $\beta_1<1$.

Next, with a view to computing the term $G_k(0\leq k\leq N+1)$, we will divide the proof into two cases $\tfrac{N}2\in\N$ and $\tfrac{N}2\not\in\N$ based on Lemma \ref{parametrix}.

{\bf Case 1: $\tfrac{N}2\in\N.$} For $\beta_1<A_0<1$, and due to the fact the even function $\tfrac{\sin h}{h}$ is analytic for $-1<h<1$, we can derive
\begin{align*}
&|G_0|\\
= & \Big|\chi\big\{\theta_1,\theta_2\in \mathbb{S}^{N-1};\;d_{\mathbb{S}^{N-1}}(\theta_1,\theta_2)<\beta_1\big\}\frac12\Big(-\frac{1}{2h}\frac{d}{dh}\Big)^{\frac{N}{2}-1}
\big(\cos h-\cos \beta_1\big)^{\frac{N}{2}-1-w}\big|_{h=d_{\mathbb{S}^{N-1}}(\theta_1,\theta_2)}\Big|\\
\lesssim& C(N)\chi\big\{\theta_1,\theta_2\in \mathbb{S}^{N-1};\;d_{\mathbb{S}^{N-1}}(\theta_1,\theta_2)<\beta_1\big\}\left(\cos d_{\mathbb{S}^{N-1}}(\theta_1,\theta_2)-\cos\beta_1\right)^{-\varepsilon}\\
\lesssim&C(N)\chi\big\{\theta_1,\theta_2\in \mathbb{S}^{N-1};\;d_{\mathbb{S}^{N-1}}(\theta_1,\theta_2)<\beta_1\big\}\left({\beta_1}^2-d_{\mathbb{S}^{N-1}}(\theta_1,\theta_2)^2\right)^{-\varepsilon}.
\end{align*}
Furthermore, by a analogous argument, we can gain
$$|G_k|\leq C\chi\left\{\theta_1,\theta_2\in \mathbb{S}^{N-1};\;d_{\mathbb{S}^{N-1}}(\theta_1,\theta_2)<\beta_1\right\},\quad\forall\;1\leq k\leq N+1,$$
which implies the estimate in \eqref{equ:t21kernesta} for $\tfrac{N}2\in\N.$

{\bf Case 2: $N=2j+1, j\in\N.$} For verifying the term $|G|$, the main difficulty lies in estimating $|G_0|$, since the function $f_w(h)=(\cos h -\cos\beta_1)^{\alpha-w}$ belongs to $C^{j-1}$, but not $C^j$. However, we first observe that
\begin{align*}
G_0= & \chi\big\{\theta_1,\theta_2\in \mathbb{S}^{N-1};\;d_{\mathbb{S}^{N-1}}(\theta_1,\theta_2)<\beta_1\big\}\\
&\times\int_{d_{\mathbb{S}^{N-1}}(\theta_1,\theta_2)}^{\beta_1}\Big[\Big(-\frac{1}{2h}\frac{d}{dh}\Big)^{j-1}
(\cos h-\cos \beta_1)^{j-\frac12-w}\Big]h\frac{(h^2-d_{\mathbb{S}^{N-1}}(\theta_1,\theta_2)^2)_+^{-\frac32}}{\Gamma(-\frac12)}\;dh.
\end{align*}
Then, we define the terms $G_{0,1}$ and $G_{0,2}$ respectively by
\begin{align*}
G_{0,1}= \int_{\frac{\beta_1+d_{\mathbb{S}^{N-1}}(\theta_1,\theta_2)}{2}}^{\beta_1}\Big[\Big(-\frac{1}{2h}\frac{d}{dh}\Big)^{j-1}
(\cos h-\cos \beta_1)^{j-\frac12-w}\Big]h\frac{(h^2-d_{\mathbb{S}^{N-1}}(\theta_1,\theta_2)^2)_+^{-\frac32}}{\Gamma(-\frac12)}\;dh,
\end{align*}
and
\begin{align*}
G_{0,2}= \int^{\frac{\beta_1+d_{\mathbb{S}^{N-1}}(\theta_1,\theta_2)}{2}}_{d_{\mathbb{S}^{N-1}}(\theta_1,\theta_2)}\Big[\Big(-\frac{1}{2h}\frac{d}{dh}\Big)^{j-1}
(\cos h-\cos \beta_1)^{j-\frac12-w}\Big]h\frac{(h^2-d_{\mathbb{S}^{N-1}}(\theta_1,\theta_2)^2)_+^{-\frac32}}{\Gamma(-\frac12)}\;dh.
\end{align*}
Therefore, we can acquire
\begin{equation}
G_0=\chi\big\{\theta_1,\theta_2\in \mathbb{S}^{N-1};\;d_{\mathbb{S}^{N-1}}(\theta_1,\theta_2)<\beta_1\big\}\big(G_{0,1}+G_{0,2}\big).
\end{equation}
Next, we turn our attentions to proving the terms $G_{0,1}$ and $G_{0,2}$.

{\bf Estimate of $|G_{0,1}|$:} This term is relatively easy, and we note that there exists a constant $C>0$ such that
$$\Big|\Big(-\frac{1}{2h}\frac{d}{dh}\Big)^{j-1}(\cos h-\cos \beta_1)^{j-\frac12-w}\Big|\leq C\big(\cos h-\cos \beta_1\big)^{\frac12-\varepsilon}$$
by the fact that $\frac{\beta_1+d_{\mathbb{S}^{N-1}}(\theta_1,\theta_2)}{2}<h<\beta_1<1$ and the even function $\tfrac{\sin h}{h}$ is analytic for $-1<h<1$.
Thus, we obtain
\begin{align*}
  |G_{0,1}|\leq & C\int_{\frac{\beta_1+d_{\mathbb{S}^{N-1}}(\theta_1,\theta_2)}{2}}^{\beta_1}\big(\cos h-\cos \beta_1\big)^{\frac12-\varepsilon}h\big(h^2-d_{\mathbb{S}^{N-1}}(\theta_1,\theta_2)^2\big)^{-\frac32}\;dh\\
  \leq&C_\ast {\beta_1}^{-\varepsilon}\left(\beta_1-d_{\mathbb{S}^{N-1}}(\theta_1,\theta_2)\right)^{-\frac32}\int_{\frac{\beta_1+d_{\mathbb{S}^{N-1}}(\theta_1,\theta_2)}{2}}^{\beta_1}
  (\beta_1-h)^{\frac12-\varepsilon}\;dh\\
  \leq&C_\ast\big({\beta_1}^2-d_{\mathbb{S}^{N-1}}(\theta_1,\theta_2)^2\big)^{-\varepsilon},
\end{align*}
where we used the property $\cos h-\cos \beta_1=2\sin\frac{\beta_1-h}{2}\sin\frac{\beta_1+h}{2}$.

{\bf Estimate of $|G_{0,2}|$:} In order to estimate $|G_{0,2}|$, we begin by recalling the general form of the Riemann-Liouville integral of order $\alpha$ given by
\begin{equation}
I^{\alpha}f(h)=\frac{1}{\Gamma(\alpha)}\int_{0}^{h}(h-u)^{\alpha-1}f(u)du,
\end{equation}
see \cite[p.11]{Rie}, with a minor adjustment in sign.
On the other hand, we write
$$g(s)=\big(-\tfrac{d}{ds}\big)^{j-1}\big(\cos\sqrt{s}-\cos \beta_1\big)^{j-\frac12-w}.$$
Therefore, performing a change of variable $s= h^2$ and applying the Riemann-Liouville integral, this results can be exhibited in the following expression
\begin{align*}
&|G_{0,2}| \\
= & \frac{1}{2\Gamma\left(-\frac{1}{2}\right)} \int_{d_{\mathbb{S}^{N-1}}^2(\theta_1, \theta_2)}^{\left(\frac{\beta_1 + d_{\mathbb{S}^{N-1}}(\theta_1, \theta_2)}{2}\right)^2}
\left[ \left( -\frac{d}{ds} \right)^{j-1} \left( \cos\sqrt{s} - \cos \beta_1 \right)^{j - \frac{1}{2} - w} \right] \left( s - d_{\mathbb{S}^{N-1}}^2(\theta_1, \theta_2) \right)^{-\frac{3}{2}} \, ds \\
= & \frac{1}{2\Gamma\left(\frac{1}{2}\right)} g \left( \left( \frac{\beta_1 + d_{\mathbb{S}^{N-1}}(\theta_1, \theta_2)}{2} \right)^2 \right)
\left[ \left( \frac{\beta_1 + d_{\mathbb{S}^{N-1}}(\theta_1, \theta_2)}{2} \right)^2 - d_{\mathbb{S}^{N-1}}^2(\theta_1, \theta_2) \right]^{-\frac{1}{2}} \\
& - \frac{1}{2\Gamma\left(\frac{1}{2}\right)} \int_{d_{\mathbb{S}^{N-1}}^2(\theta_1, \theta_2)}^{\left( \frac{\beta_1 + d_{\mathbb{S}^{N-1}}(\theta_1, \theta_2)}{2} \right)^2}
g'(s) \left( s - d_{\mathbb{S}^{N-1}}^2(\theta_1, \theta_2) \right)^{-\frac{1}{2}} \, ds \\
\leq & C_1 \left( \beta_1^2 - d_{\mathbb{S}^{N-1}}^2(\theta_1, \theta_2) \right)^{-\varepsilon} + C_2 \int_{d_{\mathbb{S}^{N-1}}^2(\theta_1, \theta_2)}^{\left( \frac{\beta_1 + d_{\mathbb{S}^{N-1}}(\theta_1, \theta_2)}{2} \right)^2}
\left( \beta_1^2 - s \right)^{-\frac{1}{2} - \varepsilon} \left( s - d_{\mathbb{S}^{N-1}}^2(\theta_1, \theta_2) \right)^{-\frac{1}{2}} \, ds \\
\leq & C \left( \beta_1^2 - d_{\mathbb{S}^{N-1}}^2(\theta_1, \theta_2) \right)^{-\varepsilon},
\end{align*}
provided that $d_{\mathbb{S}^{N-1}}(\theta_1,\theta_2)<\beta_1<1$ and ${\rm Re}w=\varepsilon\in\big(\tfrac12,1\big)$.

Hence, for $\tfrac{N}2\not\in\N$, we can achieve
$$|G_0|\leq C\chi\big\{\theta_1,\theta_2\in \mathbb{S}^{N-1};\;d_{\mathbb{S}^{N-1}}(\theta_1,\theta_2)<\beta_1\big\}\left({\beta_1}^2-d_{\mathbb{S}^{N-1}}(\theta_1,\theta_2)^2\right)^{-\varepsilon}.$$
Next, for $1\leq k\leq N+1$, we note
\begin{align*}
G_k= & \chi\big\{\theta_1,\theta_2\in \mathbb{S}^{N-1};\;d_{\mathbb{S}^{N-1}}(\theta_1,\theta_2)<\beta_1\big\}\\
&\times \int_{d_{\mathbb{S}^{N-1}}(\theta_1,\theta_2)}^{\beta_1}\Big[\Big(-\frac{1}{2h}\frac{d}{dh}\Big)^{j-1}
(\cos h-\cos\beta_1)^{j-\frac12-w}\Big]h\frac{\left(h^2-d_{\mathbb{S}^{N-1}}(\theta_1,\theta_2)^2\right)_+^{k-\frac32}}{\Gamma(k-\frac12)}\;dh.
\end{align*}
Then, similar to the method of $G_0$, it is easy to prove that
$$|G_k|\leq C\chi\big\{\theta_1,\theta_2\in\mathbb{S}^{N-1};\;d_{\mathbb{S}^{N-1}}(\theta_1,\theta_2)<\beta_1\big\}.$$
Therefore, we conclude the proof of Proposition \ref{lem:I2}.
\end{proof}

\begin{proof}[{\bf Proof of Lemma \ref{prop:I2}:}]
Now, we pay attentions to showing Proposition \ref{prop:I2} based on proposition \ref{lem:I2} established above. According to Riesz convex theorem, it suffices to confirm that there exists a constant $C(\varepsilon)>0$ such that
\begin{align*}
&\sup_{(r_1,\theta_1)\in\R^N}\int_{\R^N}\big|I_{2}\big((r_1,\theta_1),(r_2,\theta_2)\big)\big|\;dy
+\sup_{(r_2,\theta_2)\in \R^N}\int_{\R^N}\big|I_{2}\big((r_1,\theta_1),(r_2,\theta_2)\big)\big|\;dx\leq C(\varepsilon).
\end{align*}
For all $x=(r_1,\theta_1)\in \R^N$, from Lemma \ref{lem:I2}, we first denote
\begin{align}
\mathcal{H}
\triangleq
\int_{\cos \beta_1=\frac{r_1^2+r_2^2-t^2}{2r_1r_2}>\cos A_0}\int_{d_{\mathbb{S}^{N-1}}(\theta_1,\theta_2)<\beta_1}\frac{t^{2(\varepsilon-\frac{N}{2})}}{(r_1r_2)^{\varepsilon}}\left({\beta_1}^2-d_{\mathbb{S}^{N-1}}^2(\theta_1,\theta_2)\right)^{-\varepsilon}r_2^{N-1}
\;dr_2d\theta_2.
\end{align}
Note that
\begin{align*}
&\int_{d_{\mathbb{S}^{N-1}}(\theta_1,\theta_2)<\beta_1}({\beta_1}^2-d_{\mathbb{S}^{N-1}}^2(\theta_1,\theta_2))^{-\varepsilon}\;d\theta_2\\
&\leq C\int_0^{\beta_1}({\beta_1}^2-h^2)^{-\varepsilon}h^{N-2}\;dh=C(N,\varepsilon){\beta_1}^{N-1-2\varepsilon}.
\end{align*}
Hence, this yields
\begin{align*}
  \mathcal{H}\leq & C(N,\varepsilon)t^{2(\varepsilon-\frac{N}{2})}\int_{\cos \beta_1=\frac{r_1^2+r_2^2-t^2}{2r_1r_2}>\cos A_0}\left(r_1r_2\right)^{-\varepsilon}\left(\frac{t^2-(r_1-r_2)^2}{2r_1r_2}\right)^{\frac{N-1-2\varepsilon}{2}}
  r_2^{N-1}\;dr_2\\
   \leq&C(N,\varepsilon)t^{2(\varepsilon-\frac{N}{2})}\int_{\{r_1,r_2>0;
  |r_1-r_2|<t<r_1+r_2\}}\left(\tfrac{t^2-(r_1-r_2)^2}{2r_1r_2}\right)^{\frac{N-1}{2}}
  \left[t^2-(r_1-r_2)^2\right]^{-\varepsilon}r_2^{N-1}\;dr_2,
\end{align*}
provided that ${\beta_1}^2\sim\frac{t^2-(r_1-r_2)^2}{2r_1r_2}$ by the definition of $\beta_1$.

With the aim of estimating the term $\mathcal{H}$, we then divide into two cases: $r_1\geq t$ and $r_1<t$.

{\bf Case 1: $r_1\geq t.$} For $N\geq3$ and $\varepsilon\in\big(\tfrac12,1\big)$, one has
\begin{align*}
 \mathcal{H}\lesssim & C(\varepsilon)t^{2(\varepsilon-\frac{N}{2})}\int_{\cos \beta_1=\frac{r_1^2+r_2^2-t^2}{2r_1r_2}>\cos A_0}\left[t^2-(r_1-r_2)^2\right]^{\frac{N-1}{2}-\varepsilon}
 \left(\frac{r_2}{r_1}\right)^{\frac{N-1}{2}}\;dr_2\\
 \lesssim&C(\varepsilon)t^{2(\varepsilon-\frac{N}{2})}
 (t^2)^{\frac{N-1}{2}-\varepsilon}r_1\int_{\{|\frac{r_2}{r_1}-1|<\frac{t}{r_1}<
 \frac{r_2}{r_1}+1\}}\left(\frac{r_2}{r_1}\right)^{\frac{N-1}{2}}d\left(\frac{r_2}{r_1}\right)\\
 =&C(\varepsilon)\left(\frac{t}{r_1}\right)^{-1}\int_{1-\frac{t}{r_1}}^{1+\frac{t}{r_1}}
 h^{\frac{N-1}{2}}\;dh\\
 \lesssim&C(\varepsilon).
\end{align*}

{\bf Case 2: $r_1<t.$} In this case, we have $|r_1-r_2|<t<r_1+r_2$, which implies that $r_2<2t$. This, together with $N\geq3$ and $\frac{r_1^2+r_2^2-t^2}{2r_1r_2}\in(0,1)$, leads to the following estimate
\begin{align*}
   &t^{2(\varepsilon-\frac{N}{2})}\left(\frac{t^2-(r_1-r_2)^2}{2r_1r_2}\right)^{\frac{N-1}{2}}
  \left[t^2-(r_1-r_2)^2\right]^{-\varepsilon}r_2^{N-1}\\
  \lesssim&t^{2(\varepsilon-\frac{N}{2})}r_2^{N-1}\frac{t^2-(r_1-r_2)^2}{2r_1r_2}
  \left[\left(t+|r_1-r_2|\right)\cdot\left(t-|r_1-r_2|\right)\right]^{-\varepsilon}\\
  \lesssim&t^{2(\varepsilon-\frac{N}{2})}t^{(N-1)-1}\frac{1}{r_1}t^{1-\varepsilon}
  \left(t-|r_1-r_2|\right)^{1-\varepsilon}\\
  =&\left(\frac{t}{r_1}\right)^{\varepsilon-1}\frac{1}{r_1}\left(\frac{t-|r_1-r_2|}{r_1}\right)^{1-\varepsilon}\\
  \lesssim&\frac{1}{r_1}\left(\tfrac{t-|r_1-r_2|}{r_1}\right)^{1-\varepsilon}.
\end{align*}
Next, we take a variable substitution $h=\tfrac{r_2}{r_1}$, combining this with the condition $\tfrac{t}{r_1}>1$ and the fact $\varepsilon-1<0$ yields
\begin{align*}
  \mathcal{H}\lesssim &C(\varepsilon)\int_{\{
  |r_1-r_2|<t<r_1+r_2\}}\frac{1}{r_1}\left(\tfrac{t-|r_1-r_2|}{r_1}\right)^{1-\varepsilon}\;dr_2\\
  \lesssim& C(\varepsilon)\int_{1-\frac{t}{r_1}}^{1+\frac{t}{r_1}}
  \left(\tfrac{t}{r_1}-|h-1|\right)^{1-\varepsilon}\;dh\\
  \lesssim&C(\varepsilon)
\end{align*}
provided that $\frac{t}{r_1}-|h-1|\leq2$ for all $h\in\big[\frac{t}{r_1}-1,\frac{t}{r_1}+1\big].$
Therefore, we complete the proof of Lemma \ref{prop:I2}.

\end{proof}

\subsection{Proof of Lemma \ref{prop:I3}} At the end of this section, we will focus on the proof of Lemma \ref{prop:I3} following the method established in \cite{L}. For the operator $\LL_{\A}$, although we guess that the heat kernel estimate on the sphere is valid in the presence of the smooth magnetic field, it is challenging to provide an exact and detailed proof. Fortunately, the more general spectral multiplier theorem established by Sogge \cite[Theorem 5.3.1]{sogge}, offers a powerful framework that enables us to proceed with the analysis in a more tractable manner.

Let $w,\beta_1$ be as above, we define the following three functions for $s\geq\alpha:$
\begin{align}\label{equ:mwasdef}
M(w,\beta_1;s):= P_{s-\frac12}^{w-\frac{N-1}{2}}(\cos \beta_1)\sin^{\frac{N-1}{2}-w}\beta_1,
\end{align}
\begin{equation}\label{equ:uwadef}
U(w,\beta_1;s):=\frac{i2^{w-\frac{N-1}{2}}}{\sqrt{\pi}\Gamma(\frac{N}{2}-w)}
e^{-i\beta_1(w-\alpha)}s^{-\alpha}e^{-i\beta_1s},
\end{equation}
\begin{equation}\label{equ:vwadef}
V(w,\beta_1;s):=s^\alpha\int_0^1t^{w+s-\frac{N}{2}}(1-t)^{\alpha-w}(1-te^{-2i\beta_1})^{\alpha-w}\;dt.
\end{equation}
From \eqref{kernel} and definition of Legendre function $P_{a}^{b}(\cos\beta_1)$ (see \cite{GR}, \cite[Section 2]{L} and the references therein), we then obtain the relationship
\begin{equation}\label{equ:mwanaotf}
M(w,\beta_1;s)=U(w,\beta_1;s)V(w,\beta_1;s)-U(w,-\beta_1;s)V(w,-\beta_1;s).
\end{equation}
Before proceeding with the proof of Lemma \ref{prop:I3}, we first review the following lemma.
\begin{lemma}\label{lem:V}
Let $A_0\leq \beta_1<\pi$ and $w=\varepsilon+iy\big(\tfrac{N+1}2-\varepsilon\big)$ with $y\in\R$ and $\varepsilon\in\big(\tfrac12,1\big).$
Set
\begin{equation}
\Lambda\triangleq\sum_{j=0}^{[\frac{N-1}{2}]+1}\sup_{s\geq\alpha}\left|s^jV^{(j)}(w,\pm\beta_1;s)\right|
\end{equation}
Then, there exists a constant $C_\varepsilon=C(n,A_0,\varepsilon)>0$, which is independent of $\beta_1$ and $y$, such that
\begin{equation}\label{equ:lambdeest}
\Lambda\leq C_\varepsilon e^{N\pi|y|}\times \begin{cases}
(\pi-\beta_1)^{\frac12-\varepsilon},&\text{if}\quad N=3\\
1,&\text{if}\quad N\geq4.
\end{cases}
\end{equation}
\end{lemma}
\begin{proof}
The proof of this lemma is similar to that of \cite[Lemma 5.8]{L}, we omit the proof here, and refer the reader to \cite[Lemma 5.8]{L} for a more detailed procedure.
\end{proof}
\begin{lemma}[Spectral multiplier theorem \cite{sogge}]\label{spectralmulti}
Suppose $M$ be a $C^{\infty}$ compact manifold of dimension $n\geq2$ and $P=P(x,D)\in \Psi_{cl}^{1}(M)$ is self-adjoint and poositive. Fix $\beta\in C_{0}^{\infty}((\frac{1}{2},2))$ satisfying $\sum_{-\infty}^{\infty}\beta(2^j\tau)=1, \tau>0$ and let $m\in L^{\infty}(\R)$ satisfy
\begin{equation}
\sup_{\lambda>0}\lambda^{-1}\int_{-\infty}^{+\infty}\left|\lambda^{\alpha}D_{\tau}^{\alpha}(\beta(\tau/\lambda)m(\tau))\right|^2\;d\tau<\infty, \quad 0\leq\alpha\leq\kappa,
\end{equation}
where $\kappa$ is an integer $\kappa>\frac{n}{2}$. Then, we obtain
\begin{equation}
\|m(P)\|_{L^p(M)}\leq C_{p}\|f\|_{L^p(M)}, \quad 1<p<\infty.
\end{equation}
\end{lemma}
On the one hand, we set
$$V_\ast(w,\beta_1;h)=V(w,\beta_1;\sqrt{\alpha^2+h^2}),\quad h\geq0,$$
hence, we have
$$V(w,\beta_1;L)=V_\ast(w,\beta_1;\sqrt{\LL_{\A}}).$$
From Lemma \ref{lem:V} and Lemma \ref{spectralmulti}, we observe
that there exist two constants $C'>C>0$ such that for all $1<p<+\infty,$
\begin{align}\label{equ:V}
\|V(w,\pm \beta_1;L)\|_{p\to p}= & \big\|V_\ast(w,\pm \beta_1;\sqrt{\LL_{\A}}\big\|_{p\to p}\leq C_{p}\Lambda.
\end{align}
On the other hand, according to the result in \cite[Corollary 4.3.2]{sogge} and \cite[p.184]{sogge}, we can derive that for all $1<p<+\infty$, there exists a constant $C_p>0$ such that for all $0<\beta_1<\pi$, ${\rm Re}w\in\big(\tfrac12,1\big)$, and  $g\in L^p(\mathbb{S}^{N-1})$, there holds
\begin{equation}\label{equ:uwaest}
  \big\|U(w,\pm \beta_1;L)g\big\|_{L^p(\mathbb{S}^{N-1})}\leq\frac{C_p}{|\Gamma(\frac{N}{2}-w)|}e^{\pi|{\rm Im}w|}\|g\|_{L^p(\mathbb{S}^{N-1})}.
\end{equation}
Therefore, by \eqref{equ:mwanaotf}-\eqref{equ:uwaest}, we can acquire
\begin{lemma}\label{prop:5.9}
Let $A_0\leq \beta_1<\pi$ and $w=\varepsilon+iy\big(\tfrac{N+1}2-\varepsilon\big)$ with $y\in\R$ and $\varepsilon\in\big(\tfrac12,1\big).$ Let $1<p<+\infty$, then there exists a constant $C_\varepsilon=C(n,A_0,\varepsilon)>0$, which is independent of $\beta_1$ and $y$, such that
\begin{equation}\label{equ:mwaleest}
\big\|M(w,\beta_1;L)\big\|_{p\to p}\leq C_p(\varepsilon)\frac{ e^{2N\pi|y|}}{|\Gamma(\frac{N}{2}-w)|}\times \begin{cases}
(\pi-\beta_1)^{\frac12-\varepsilon},&\text{if}\quad N=3\\
1,&\text{if}\quad N\geq4.
\end{cases}
\end{equation}
\end{lemma}
Finally, for the purpose of proving Lemma \ref{prop:I3}, it suffices to verify the following Lemma by the Schur's test lemma.
\begin{lemma}\label{lem:aim}
Set
\begin{align*}
K_w(r_1,r_2)= & \chi\Big\{-1<\cos\beta_1=\frac{r_1^2+r_2^2-t^2}{2r_1r_2}\leq \cos A_0\Big\}\\
&\times t^{2(w-\frac{N}{2})}(r_1r_2)^{-w}\begin{cases}
(\pi-\beta_1)^{\frac12-\varepsilon},&\text{if}\quad N=3\\
1,&\text{if}\quad N\geq4.
\end{cases}
\end{align*}
Then, there exists a constant $C(\varepsilon)>0$ such that
$$\sup_{r_1>0}\int_0^{+\infty}\big|K_w(r_1,r_2)\big|r_2^{N-1}\;dr_2=
\sup_{r_2>0}\int_0^{+\infty}\big|K_w(r_1,r_2)\big|r_1^{N-1}\;dr_1<C(\varepsilon).$$
\end{lemma}
\begin{proof}
Before concluding the proof of Lemma \ref{lem:aim}, we first emphasize the conditions  $\frac{1}{2}<\varepsilon<1$ and $0<\beta_1=\arccos\tfrac{r_1^2+r_2^2-t^2}{2r_1r_2}<\pi$ in this article, which will be applied frequently in subsequent proofs.

Our goal is to show that there exists a constant $C(\varepsilon)>0$ such that for all $r_1>0$ and $-1<\cos \beta_1=\frac{r_1^2+r_2^2-t^2}{2r_1r_2}\leq \cos A_0$, we can establish
\begin{align}\label{equ:aimest}
\mathcal{F}\triangleq\int
t^{2(\varepsilon-\frac{N}{2})}(r_1r_2)^{-\varepsilon}(\pi-\beta_1)^{-2\varepsilon}
r_2^{N-1}\;dr_2<C(\varepsilon).
\end{align}
Combining the fact
$$\pi-\beta_1=2\tfrac{\pi-\beta_1}{2}\sim \sin\tfrac{\pi-\beta_1}{2}=2\sqrt{\tfrac{1-\cos(\pi-\beta_1)}2}=2\big(\tfrac{(r_1+r_2)^2-t^2}{4r_1r_2}\big)^\frac12$$
with
$$\tfrac{(r_1+r_2)^2-t^2}{2r_1r_2}\leq \cos A_0,$$
we can acquire
$$\big(1-\tfrac{r_2}{r_1}\big)^2+2\tfrac{r_2}{r_1}(1-\cos A_0)\leq\big(\tfrac{t}{r_1}\big)^2.$$
Therefore, there exists a constant $C(A_0)>0$ such that $r_1\leq C(A_0)t.$
Furthermore, using the same approach, we can deduce that $r_2\leq C(A_0)t$.

As a consequence, by changing variable $h=\tfrac{r_2}{r_1}-\big(\tfrac{t}{r_1}-1\big)$, we can compute
\begin{align*}
  \mathcal{F}\lesssim & C(A_0)\int
t^{2(\varepsilon-\frac{N}{2})}\big[(r_1+r_2)^2-t^2\big]^{-\varepsilon}
r_2^{N-1}\;dr_2\\
\lesssim&C(A_0)\int_{|r_2-r_1|<t<r_1+r_2}\big(\tfrac{r_1+r_2-t}{r_1}\big)^{-\varepsilon}
d\big(\tfrac{r_2}{r_1}\big)\\
=&C(A_0)\int_0^2 h^{-\varepsilon}\;dh\lesssim C(A_0)\frac{1}{1-\varepsilon},
\end{align*}
which implies \eqref{equ:aimest} and finishes the proof of Lemma \ref{lem:aim}.
\end{proof}
\section{The proof of Theorem \ref{Lpestimate}}
In this section, we focus on the proof of Theorem \ref{Lpestimate}. Inspired by the work of \cite{L}, the main idea is primarily based on the asymptotic properties of Bessel functions and results related to spectral multipliers within the context of magnetic fields for higher dimensions. Also, we divide the proof into three steps.

\textbf{Step 1: Spectral multiplier theorem.}
The spectral multiplier theorem serves as the foundational result for demonstrating Theorem \ref{Lpestimate}. Notably, from the Simon's diamagnetic inequality (see \cite{AHS} and \cite[Theorem B.13.2]{Simon}), the heat kernel estimate for magnetic fields can be derived directly, then we can obtain the spectral multiplier theorem below by \cite[Theorem 7.23]{Ouhabaz} and \cite[Theorem 2]{SW}.
\begin{lemma}\label{lem:multiplier}
Let $\LL_{\A}$ be the magnetic Schr\"odinger operator in \eqref{LA}.

$\bullet$ $L^p$-bounds for imaginary powers operator.  For all $y\in\R$, then the imaginary powers $(\LL_{\A})^{iy}$ satisfy the $(1,1)$ weak type estimate
\begin{equation}\label{im-bounds}
\big\|(\LL_{\A})^{iy}\big\|_{L^1(\R^N)\to L^{1,\infty}(\R^N)}\leq C(1+|y|)^{N/2}.
\end{equation}
In particular,  there exists a constant $C>0$ such that
\begin{equation}\label{spectral-LA}
\|(\LL_{\A})^i\|_{L^p(\R^N)\to L^{p}(\R^N)} \leq C, \quad\forall 1<p<+\infty.
\end{equation}

$\bullet$ Mikhlin-H\"ormander multiplier. Assume that $g:[0,+\infty) \rightarrow \mathbb{R}$ satisfies
\begin{equation}
\sum_{j=0}^{[\frac{N}{2}+1]} \sup _{s \geq 0} s^j|g^{(j)}(s)|<+\infty,
\end{equation}
then, for all $1<p<+\infty$ and $t>0$,
\begin{equation}\label{estimate:spectraltLA}
\|g(t \sqrt{\LL_{\A}})\|_{L^p(\R^N)\to L^{p}(\R^N)} \leq C \sum_{j=0}^{[\frac{N}{2}+1]} \sup _{s \geq 0} s^j|g^{(j)}(s)| .
\end{equation}

$\bullet$ Let $\sigma>\frac{N}{2}$ and $g:[0,+\infty) \to \mathbb{R}$ obeys
\begin{equation}\label{estimate-condition}
|g^{(j)}(s)| \leq C (1+s)^{-\sigma}, \quad\forall s \geq 0 \text { and } 0 \leq j \leq [\frac{N}{2}]+1,
\end{equation}
then there exists a constant $C(\sigma)>0$ such that
\begin{equation}\label{ResultIII}
\|g(t \sqrt{\LL_{\A}})\|_{L^p(\R^N)\to L^{p}(\R^N)} \leq C(\sigma), \quad \forall 1<p<+\infty, \ t>0 .
\end{equation}
\end{lemma}
\textbf{Step 2: Division of the proof interval.} To show Theorem \ref{Lpestimate}, it suffices to verify that for all $\frac{N}{4}>\ell>\frac{N-1}{4}$ and $1< p<+\infty$, there exists a constant $C(p, \ell)>0$ such that for all $f \in L^p\left(\mathbb{R}^N\right)$ and $t>0$
\begin{equation}\label{estimate-aim}
\left\|\left(1+\mathcal{L}_{\mathbf{A}}\right)^{-\ell} e^{i t \sqrt{\mathcal{L}_{\mathbf{A}}}} f\right\|_{L^p(\R^N)} \leq C(p, \ell)(1+t)^{2 \ell}\|f\|_{L^p(\R^N)},
\end{equation}
which corresponds to the region $\Omega_{2}$.
In fact, once we establish that the region $\Omega_{2}$ is concluded, the other two regions can be determined immediately. Next, we will prove how the other two regions can be derived from this part.

When $\ell \geq \frac{N}{4}$, it is evident that the operator $\left(1+\mathcal{L}_{\mathbf{A}}\right)^{-\left(\ell-\frac{2N-1}{8}\right)}$ is bounded in $L^p\left(\mathbb{R}^N\right)$ for $1< p<+\infty$, with the norm that is independent of $p$. Combining with \eqref{estimate-aim} yields Theorem \ref{Lpestimate} for region $\Omega_{3}$.

For $\ell \leq \frac{N-1}{4}$, the point $A$ is trivial. Therefore, we utilize Stein's interpolation theorem (see \cite{stein}) to interpolate between the point $A$ and region $\Omega_{2}$, and then obtain Theorem \ref{Lpestimate} in region $\Omega_{1}$.
\begin{remark}
As a matter of fact, the partition at $l\geq\frac{N}{4}$ is not strictly necessary, and $\frac{N}{4}$ can be substituted with any number larger than $\frac{N-1}{4}$. Accordingly, the boundedness of \eqref{estimate-aim} can also be altered to $(1+t)^{\frac{N-1}{2}+\varepsilon}$ for all $\varepsilon>0$.
\end{remark}

\begin{figure}[htbp]
\begin{center}
\scalebox{0.9}[0.9]{
\begin{tikzpicture}
\draw (0,0) rectangle (6.75,6.75);
\draw[->]  (0,0) -- (0,7.5);
\draw[->]  (0,0) -- (7.5,0);
\draw (6.75,0) node[below] {$\ 1$};
\draw (0,0) node[below, left] {$0$};
\draw (7.5,0) node[above] {$1/p$};
\draw (0,7.5) node[right] {$l$};
\draw (3.35,0) circle (2pt) node[left] {$A$};    
\draw (6.75,1.7)  circle (2pt) node[right] {$B$}; 
\draw (0,1.7)  circle (2pt) node[left] {$C$}; 
\draw (6.75,3.35)  circle (2pt) node[right] {$\!D$}; 
\draw (0,3.35)  circle (2pt) node[left] {$\!E$}; 
\draw[dashed] (3.35,0) -- (0,1.7) ; 
\draw[dashed] (3.35,0) -- (6.75,1.7) ; 
\draw[dashed] (0,1.7) -- (6.75,1.7) ; 
\draw[dashed] (0,3.35) -- (6.75,3.35) ; 
\draw (3.35,0) node[below] {$\frac{1}{2}$};
\filldraw[fill=gray!50][dashed](0,1.7)--(6.75,1.7)--(3.35,0)--cycle; 
\filldraw[fill=blue!50](6.75,3.35)--(6.75,1.7)--(0,1.7)--(0,3.35);
\filldraw[fill=red!50](6.75,3.35)--(6.75,6.75)--(0,6.75)--(0,3.35);
\filldraw[fill=black](3.15,4.5) node[above right] {$\Omega_{3}$};
\filldraw[fill=black](3.15,2.25)  node[above right] {$\Omega_{2}$};
\filldraw[fill=black](3.15,0.7) node[above right] {$\Omega_{1}$};
\end{tikzpicture}
}
\end{center}
\caption{Here $A=(\frac{1}{2},0)$, $B=(1,\frac{N-1}{4})$, $C=(0,\frac{N-1}{4})$, $D=(1,\frac{N}{4})$, $E=(0,\frac{N}{4})$, respectively. The line $AB: l=\frac{N-1}{2}(\frac{1}{p}-\frac{1}{2})$. The line $AC: l=\frac{N-1}{2}(\frac{1}{2}-\frac{1}{p})$.}  \label{fig:1}
\end{figure}
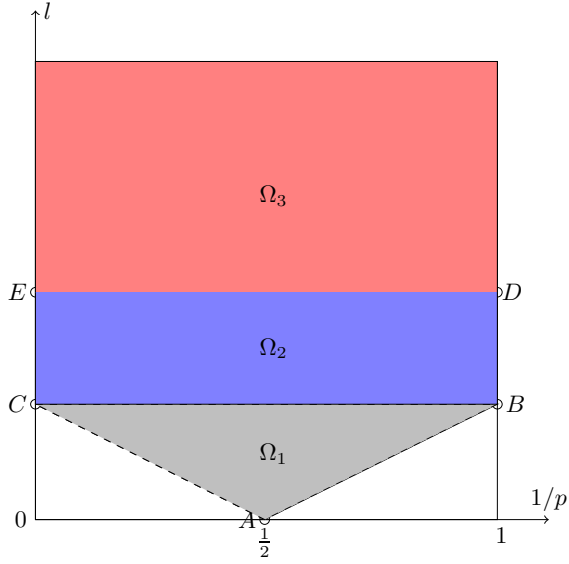

\textbf{Step 3: The proof of \eqref{estimate-aim}.} First, for all $s>0$, we define the following functions
$$
\begin{aligned}
&m(\ell, s)=\left(1+s^2\right)^{-\ell} e^{i s},\\
& m_{\ell}(s)=\psi(s) s^{-2 \ell} e^{i s}, \\
& M_{\ell}(s)=m(\ell, s)-m_{\ell}(s),
\end{aligned}
$$
in which $\psi \in C^{\infty}\left(\mathbb{R}_{+}\right)$ satisfies the following conditions:
\[
\psi(s) =
\begin{cases}
0, & \text{if } s \leq 1, \\
1, & \text{if } s \geq 2.
\end{cases}
\]
Then, from \eqref{estimate:spectraltLA} in Lemma \ref{lem:multiplier}, we attain
$$
\begin{aligned}
\Big\|\left(1+\mathcal{L}_{\mathbf{A}}\right)^{-\ell} e^{i t \sqrt{\mathcal{L}_{\mathbf{A}}}}\Big\|_{p \rightarrow p} & \leq\Big\|\frac{\left(1+t^2 \mathcal{L}_{\mathbf{A}}\right)^{\ell}}{\left(1+\mathcal{L}_{\mathbf{A}}\right)^{\ell}}\Big\|_{p \rightarrow p}\times\Big\|\left(1+t^2 \mathcal{L}_{\mathbf{A}}\right)^{-\ell} e^{i t \sqrt{\mathcal{L}_{\mathbf{A}}}}\Big\|_{p \rightarrow p} \\
& \leq C \sum_{j=0}^2 \sup _{s \geq 0}\Big|s^j \frac{d^j}{d s^j} \frac{\left(1+s^2\right)^{\ell}}{\left(1+s^2/t^2\right)^{\ell}}\Big| \times \Big\|\left(1+t^2 \mathcal{L}_{\mathbf{A}}\right)^{-\ell} e^{i t \sqrt{\mathcal{L}_{\mathbf{A}}}}\Big\|_{p \rightarrow p}\\
&\lesssim (1+t)^{2\ell}\Big\|\left(1+t^2 \mathcal{L}_{\mathbf{A}}\right)^{-\ell} e^{i t \sqrt{\mathcal{L}_{\mathbf{A}}}}\Big\|_{p \rightarrow p}.
\end{aligned}
$$
Therefore, the proof of \eqref{estimate-aim} can be reduced to showing that for all $\frac{N}{4}>\ell>\frac{N-1}{4}$ and $1<p<+\infty$, there exists a constant $C(\ell, p)>0$ such that
\begin{equation}\label{estimate-mlt}
\|m(\ell, t \sqrt{\mathcal{L}_{\mathbf{A}}})\|_{p \rightarrow p} \leq C(\ell, p), \quad \forall t>0 .
\end{equation}

Consequently, combining with the definition of $m(\ell, t \sqrt{\mathcal{L}_{\mathbf{A}}})$, it remains to achieve the following proposition. The proof of Proposition \ref{estimate:ml} follows the general outline in \cite[Theorem 1.2]{L} and we omits the detailed proof.
\begin{proposition}\label{estimate:ml}
Let $\frac{N}{4}>\ell>\frac{N-1}{4}$ and $1<p<+\infty$. Then, there exists a constant $C(p, \ell)>0$ such that
\begin{equation}\label{estimate-ml}
\|m_{\ell}(t \sqrt{\mathcal{L}_{\mathbf{A}}}) f\|_{L^p(\R^N)} \leq C(p, \ell)\|f\|_{L^p(\R^N)}, \quad \forall f \in L^p\left(\mathbb{R}^N\right), t>0,
\end{equation}
and
\begin{equation}\label{estimate:Mls1}
\|M_{\ell}(t \sqrt{\mathcal{L}_{\mathbf{A}}}) f\|_{L^p(\R^N)} \leq C(\ell)\|f\|_{L^p(\R^N)}, \quad \forall f \in L^p\left(\mathbb{R}^N\right), t>0.
\end{equation}
\end{proposition}

\section{The proof of Theorem \ref{sinLA}}
In this section, we aim to show Theorem \ref{sinLA} by utilizing the $p\rightarrow p$ boundedness of the analytic operator family $f_{w,t}(\LL_{\A})$ and Stein's interpolation theorem.

First, based on the result from \cite[p.463]{Stri}, we observe that
$$\big|f_{\frac{N+1}{2}+i(\frac{N+1}{2}-\varepsilon)s,t}(h)\big|
\leq C_1e^{C_2(\frac{N+1}{2}-\varepsilon)|s|},\quad \forall\;t,h>0,$$
where we define $w=\varepsilon+(1+is)(\frac{N+1}{2}-\varepsilon)$ with $\tfrac12<\varepsilon<1$.
Thus, we can obtain
\begin{equation}\label{equ:2to2}
\big\|f_{\frac{N+1}{2}+i(\frac{N+1}{2}-\varepsilon)s,t}(\LL_{\A})\big\|_{L^2\to L^2}
\leq C_1e^{C_2(\frac{N+1}{2}-\varepsilon)|s|},
\end{equation}
in which $C_1$ and $C_2$ are two constants that are independent of the other variables. Moreover, with Theorem \ref{operator-family} in hand, we know
\begin{equation}\label{est:fw}
\left\|f_{w,t}(\mathcal{L}_{\mathbf{A}})g\right\|_{L^p(\R^N)}\leq C(\varepsilon,p)\frac{e^{-2N\pi|y|}}{|\Gamma(\frac{N}{2}-w)|}\left\| g\right\|_{L^p(\R^N)}, \quad\forall\; g\in L^{p}(\mathbb{R}^{N})
\end{equation}
for $w=\varepsilon+iy(\frac{N+1}{2}-\epsilon)$ with $y\in \mathbb{R}$ and $1<p<\infty$.\\
Note that
$$f_{\frac{N-1}{2},t}(h)=\frac{\sin(t\sqrt{h})}{t\sqrt{h}}.$$
Applying the Stein's interpolation theorem to \eqref{equ:2to2} and \eqref{est:fw} with $p=1$, there exists a constant $C(p(\varepsilon))>0$ such that
$$\Big\|\frac{\sin(t\sqrt{\LL_{\A}})}{t\sqrt{\LL_{\A}}}\Big\|_{p(\varepsilon)
\to p(\varepsilon)}\leq C(p(\varepsilon)).$$
in which $p(\varepsilon)$ will be chosen later.
Indeed, we define $\theta(\varepsilon)\in(0,1)$ by
$$\frac{N-1}{2}=\frac{N+1}{2}\theta(\varepsilon)+\varepsilon(1-\theta(\varepsilon)),$$
and
$$\frac{1}{p(\varepsilon)}=\big(1-\theta(\varepsilon)\big)+\frac{\theta(\varepsilon)}{2}
=1-\frac{\theta(\varepsilon)}{2}\quad\text{or}\quad
\frac{1}{p(\varepsilon)}=\frac{\theta(\varepsilon)}{2}.$$
Then, we can acquire
$$\tfrac1{p(\varepsilon)}\in\big(\tfrac12\tfrac{N-3}{N-1},\tfrac12\tfrac{N-2}N\big)\cup\big(\tfrac12+\tfrac1N,
\tfrac12+\tfrac1{N-1}\big).$$

Finally, employing Riesz's convexity theorem, we can achieve Theorem \ref{sinLA}.

\begin{center}

\end{center}
 \end{document}